\documentclass[11pt]{amsart}
\usepackage[a4paper]{geometry}
\usepackage{times}
\usepackage[english]{babel}
\usepackage[utf8]{inputenc}
\usepackage[pdftex]{color, graphicx}
\usepackage{amsmath,amsthm}
\usepackage{amssymb,amsfonts}
\usepackage[foot]{amsaddr}
\usepackage[hang,flushmargin]{footmisc} 
\usepackage{mathrsfs}
\usepackage{tikz}
\usepackage{graphicx}
\usepackage{enumerate}
\usepackage{xypic}
\usepackage{microtype}
\usepackage{hyperref}
\usepackage{cleveref}
\usepackage{amscd}
\usepackage{booktabs}
\usepackage{todonotes}
\setuptodonotes{inline, color=pink}

\usepackage{mathtools}
\DeclarePairedDelimiter{\set}{\{}{\}}
\DeclarePairedDelimiter{\abs}{\lvert}{\rvert}
\usepackage{centernot}

\usepackage[autostyle, italian=guillemets]{csquotes}

\newcommand{\IMP}{\Rightarrow}
\newcommand{\IFF}{\Leftrightarrow}

\newcommand{\tp}{\textrm{tp}}

\newcommand{\BD}{\mathrm{BD}}
\newcommand{\Z}{\mathbb{Z}}
\newcommand{\N}{\mathbb{N}}

\newcommand{\cc}{\mathsf{c}}
\newcommand{\tmid}{\mathrel{\tilde \mid}}
\newcommand{\smid}{\mathrel{\mid^\mathrm{s}}}
\newcommand{\nsmid}{\mathrel{\centernot\mid^\mathrm{s}}}
\newcommand{\sequiv}[1]{\mathrel{\equiv^\mathrm{s}_{#1}}}
\newcommand{\nsequiv}[1]{\mathrel{\not\equiv^\mathrm{s}_{#1}}}
\newcommand{\DF}{\mathrm{SD}}
\newcommand{\DL}{\mathrm{DL}}
\newcommand{\from}{\colon}
\newcommand{\inverse}{^{-1}}
\renewcommand{\phi}{\varphi}

\theoremstyle{plain}
\newtheorem{theorem}{Theorem}[section]
\newtheorem*{utheorem}{Main Theorem}
\newtheorem{lemma}[theorem]{Lemma}
\newtheorem{fact}[theorem]{Fact}
\newtheorem{corollary}[theorem]{Corollary}

\newtheorem{proposition}[theorem]{Proposition}

\theoremstyle{definition}
\newtheorem{definition}[theorem]{Definition}

\newtheorem{example}[theorem]{Example}
\newtheorem{problem}[theorem]{Problem}
\theoremstyle{remark}
\newtheorem{remark}[theorem]{Remark}

\usepackage{fancyhdr}
\usepackage{orcidlink}

\usepackage{url}
\makeatletter                        
\g@addto@macro{\UrlBreaks}{\UrlOrds} 
\makeatother                         
\title{Self-divisible ultrafilters and congruences in $\beta \mathbb Z$}

\author{Mauro Di Nasso$^\dagger$\,\orcidlink{0000-0001-6103-9775}}
\address{$^\dagger$ Dipartimento di Matematica, Universit\`a di Pisa, Largo Bruno Pontecorvo 5, 56127 Pisa, Italy}
\author{Lorenzo Luperi Baglini$^*$\,\orcidlink{0000-0002-0559-0770}}
\address{$^*$ Dipartimento di Matematica, Università  di Milano, Via Saldini 50, 20133 Milano, Italy}
\author{Rosario Mennuni$^\dagger$\,\orcidlink{0000-0003-2282-680X}}
\author{Moreno Pierobon$^\dagger$}
\author{Mariaclara Ragosta$^\dagger$\,\orcidlink{0009-0004-6641-4676}}

\keywords{congruence, divisibility, ultrafilters, tensor product, nonstandard integers, profinite integers}
\subjclass[2020]{Primary: 54D80, 11A07. Secondary, 03H15, 03C62, 20E18}
\begin{document}
\begin{abstract}
  We introduce \emph{self-divisible} ultrafilters, which we prove to be precisely those $w$ such that  the weak congruence relation $\equiv_w$ introduced by \v Sobot is an equivalence relation on $\beta \mathbb Z$. We provide several examples and additional characterisations; notably we show that $w$ is self-divisible if and only if $\equiv_w$ coincides with the strong congruence relation $\sequiv w$, if and only if  the quotient   $(\beta \mathbb Z,\oplus)/\mathord{\sequiv w}$ is a profinite group.
  We also construct an ultrafilter $w$ such that $\equiv_w$ fails to be symmetric, and describe the interaction between the aforementioned quotient and the profinite completion $\hat{\mathbb Z}$ of the integers.
\end{abstract}
\maketitle

In~\cite{sobot_congruence_2021}, \v Sobot investigated generalisations of the congruence relation $a\equiv_n b$ from $\mathbb Z$ to its  Stone--\v Cech compactification $\beta \mathbb Z$, equipped with the usual extensions $\oplus$, $\odot$ of the sum and product of integers. For each $w\in \beta \mathbb N$, he introduced a \emph{congruence} relation $\equiv_w$ and a \emph{strong congruence} relation $\sequiv w$. In this paper we investigate these notions and prove that, for some $w$, the former fails to be an equivalence relation, thereby answering~\cite[Question~7.1]{sobot_congruence_2021} in the negative. In fact, we fully characterise those $w$ for which this happens, and compute the relative quotient when it does not.

Almost by definition, $u\sequiv w v$ holds if and only if, whenever $(d,a,b)$ is an ordered triple of nonstandard integers which generate $w\otimes u\otimes v$, we have $d\mid a-b$. It was proven in~\cite{sobot_congruence_2021} that $\sequiv w$ is always an equivalence relation, and in fact a congruence with respect to $\oplus$ and $\odot$ but, perhaps counterintuitively for a notion of congruence, there are some ultrafilters $w$ for which $w\nsequiv w 0$. On the other hand, the relation $\equiv_w$ does always satisfy $w\equiv_w 0$ but, as we  said above, it may fail to be an equivalence relation. So, in a sense, these two relations have complementary drawbacks, and it is natural to ask for which ultrafilters $w$ these drawbacks disappear.
Our main result says that $\equiv_w$ is well-behaved if and only if $\sequiv w$ is, if and only if the two relations collapse onto each other. This is moreover equivalent to the quotient $(\beta \mathbb Z, \oplus)/\mathord{\sequiv w}$ being a profinite group, which in fact can be explicitly computed.
More precisely, if we denote by $\mathbb P$ the set of prime natural numbers and by $\mathbb Z_p$ the additive group of $p$-adic integers, our main results can be summarised as follows.
\begin{utheorem}[\Cref{main:thm,thm:charred}]
For every $w\in \beta \mathbb N$ the following are equivalent.
\begin{enumerate}
\item \label{point:mainthmintro1}We have $w\sequiv w0$.
\item The relation $\equiv_w$ is an equivalence relation.
\item \label{point:mainthmintro3}The relations $\equiv_w$ and $\sequiv w$ coincide.
\item The quotient $(\beta \mathbb Z, \oplus)/\mathord{\sequiv w}$ is isomorphic to $\prod_{p \in \mathbb P} G_{p,w}$, where 
\[G_{p,w}=\begin{cases}
\Z/ p^n \mathbb Z, & \text{if} \ n=\max\{k\in\mathbb{N}\cup\{0\}: p^k \mathbb Z\in w\} \text{ exists};\\
\Z_{p}, & \text{otherwise}.\end{cases}\]
\item The quotient $(\beta \mathbb Z, \oplus)/\mathord{\sequiv w}$, when equipped with the quotient topology, is a profinite group.
\end{enumerate}
\end{utheorem}
In fact, many more characterisations of those $w$  satisfying the  above equivalent conditions are possible (\Cref{thm:charred}).  We believe this  to be an indication that these ultrafilters, which we dub \emph{self-divisible}, are objects of interest, and we study them at length throughout the paper.

In more detail, after briefly recalling the context in \Cref{sec:prelim}, we prove in \Cref{sec:negative} that  $\equiv_w$ is an equivalence relation if and only if it is transitive and provide an example of an ultrafilter $w$ such that $\equiv_w$ is not symmetric.  \Cref{sec:main} is devoted to proving the equivalence of~\eqref{point:mainthmintro1} to~\eqref{point:mainthmintro3} in the Main Theorem. In \Cref{sec:eg} we provide examples of self-divisible ultrafilters, and study the topological properties of their space in \Cref{sec:topology}. In \Cref{sec:algebra} we  study the quotients $\beta \mathbb Z/\mathord{\sequiv w}$, and show that each of them may be identified with a quotient of the profinite completion $\hat{\mathbb Z}$ of the integers which embeds in the ultraproduct $\prod_w \mathbb Z/n \mathbb Z$. In the same section, we obtain several further equivalent definitions of self-divisibility, completing the proof of the Main Theorem. We conclude in \Cref{sec:craaop} with some final remarks, further directions, and an open problem.

\section{Preliminaries}\label{sec:prelim}
The letters $u,v,w,t$ will usually denote elements of $\beta\mathbb Z$,  while $p,q,r$ will typically stand for prime natural numbers. We identify each integer with the corresponding principal ultrafilter.
If $u\in \beta \mathbb Z$ we write $-u$ for $\set{-A: A\in u}$ and $u\ominus v$ for   $u\oplus (-v)$. We extend the usual conventions about usage of $+$, $-$ to $\oplus$, $\ominus$, e.g.\ whenever we write $-u\oplus v$ we mean $(-u)\oplus v$, and $u\ominus v\ominus w$ is to be parsed as $u\oplus (-v)\oplus (-w)$. If $A\subseteq \mathbb Z$, then $A^\cc$ denotes $\mathbb Z\setminus A$, and $\overline A$ denotes the closure of $A$ in $\beta \mathbb Z$, that is, $\set{u\in \beta \mathbb Z : A\in u}$. We convene that $0\notin \mathbb N$, and use $\omega$ for $\mathbb N\cup \set 0$.

We adopt some conventions and notations of model-theoretic flavour; some standard references are~\cite{hodges,poizat,tent-ziegler}. Namely, we work in a $\kappa$-saturated elementary extension ${}^{\ast}\mathbb Z$ of $\mathbb Z$, where the latter is equipped with a symbol for every subset of every cartesian power $\mathbb Z^k$, and where $\kappa$ is a large enough cardinal, for instance $\kappa=(2^{\aleph_0})^+$. The results obtained do not depend on the particular elementary extension chosen. Moreover, we write $a\models u$, or $u=\tp(a/\mathbb Z)$, and say that $a$ is a \emph{realisation} of $u$, or that $a$ \emph{generates} $u$, to mean that $u=\set{A\subseteq \mathbb Z: a\in{}^{\ast} A}$. In other words, we identify ultrafilters in $\beta \mathbb Z$ with $1$-types over $\mathbb Z$ in the language mentioned above.

In this setting, every type over $\mathbb Z$ is definable, and the product $\otimes$ of
ultrafilters coincides with the product $\otimes$ of definable types, provided compatible conventions are adopted. Specifically, we have $A\in u\otimes v\IFF\set{x: \set{y: (x,y)\in A}\in v}\in u$.  In terms of realisations, this means that the order in which we resolve tensor products is reversed with respect to the majority of model-theoretic literature; namely, in this paper  $(a,b)\models u\otimes v$ iff $a\models u$ and $b\models v\mid \mathbb Z a$.\footnote{Here $v\mid \mathbb Z a$ denotes the unique type over $\mathbb Z a$ extending $v$ and definable over $\mathbb Z$.} In this case, we call $(a,b)$ a \emph{tensor pair}. Tensor pairs in ${}^{\ast}\mathbb{N}$ have been characterised by Puritz in~\cite[Theorem~3.4]{PUR}; we recall here the extension of Puritz' characterisation to ${}^{\ast}\mathbb{Z}$, and we refer to~\cite[Section~11.5]{Di_Nasso_2015} or~\cite{LMON} for a proof of this fact.

\begin{fact}\label{fact:abstensorpair}
An ordered pair $(a,b)\in{}^{\ast}\Z^{2}$ is a tensor pair if and only if for every $f\from\Z\to\Z$ either ${}^{\ast}f(b)\in\Z$ or $\abs a\leq\abs{{}^{\ast}f(b)}$.  
\end{fact}

The iterated hyper-extensions framework of nonstandard analysis allows for an even simpler characterisation of tensor products and related notions: if $a,b\in{}^{\ast}\mathbb{Z}$ are such that $a\models u$ and  $b\models v$, then $\left(a,{}^{\ast}b\right)\models u\otimes v$. As a trivial consequence, in the same hypotheses we have that $a+{}^{\ast}b\models u\oplus v$ and $a\cdot {}^{\ast}b\models u\odot v$. A detailed study of many properties and characterisations of tensor $k$-uples in this iterated nonstandard context can be found in~\cite{LMON}.

Let us recall (some equivalent forms of) the definitions of divisibility and congruence of ultrafilters. We will frequently use that, when dealing with generators of ultrafilters, some existential quantifiers may be replaced by universal ones. For example, $(\exists a\models u)\; (\exists b\models v)\; a\mid b$ if and only if $(\forall a\models u)\; (\exists b\models v)\; a\mid b$, if and only if $(\forall b\models v)\; (\exists a\models u)\; a\mid b$. This follows from saturation of ${}^{\ast}\mathbb Z$, see~\cite[Corollary~5.13]{LMON}.  By this, and~\cite[Proposition~3.2 and Theorem~4.5]{sobot_congruence_2021}, we may take as definitions of $\tmid$ and $\equiv_w$ the ones below.\footnote{In \cite{sobot_congruence_2021} the relation $\equiv_w$ is only defined for $w\in \beta\mathbb N$. Clearly there is no harm in using the same definition for $w\in \beta\mathbb Z\setminus \set 0$, but at any rate it is immediate that $\equiv_w$ coincides with $\equiv_{-w}$, and similarly for $\sequiv w$.} Similarly, our definition of $\sequiv w$ is not the original one, but it is equivalent to it by \cite[Lemma~6.5]{sobot_congruence_2021}.

\begin{definition}\label{defin:midequiv}
Let $u,v,w\in \beta \mathbb Z$, with $w\ne 0$.
  \begin{enumerate}
  \item   We write $u\tmid v$ iff there are $a\models u$ and $b\models v$ such that $a\mid b$.
  \item We write $u\equiv_w v$ iff there are $d\models w$ and $(a,b)\models u\otimes v$ such that $d\mid a-b$.
  \item   We write $u\smid v$ iff there is  $(a,b)\models u\otimes v$ such that $a\mid b$.
      \item   We write $u\sequiv w v$ iff there is  $(d,a,b)\models w\otimes u\otimes v$ such that $d\mid a-b$.
  \end{enumerate}
\end{definition}
We stress that  the existential quantifier in the definition of $\smid$ may be replaced with a universal one: the property being checked is true of \emph{some} realisation of the tensor product if and only if it is  true of \emph{every} realisation of the tensor product. The same holds for $\sequiv w$, but not for $\tmid$, nor for $\equiv_{w}$: in the latter two cases, we can replace one (any) existential quantifier with an universal one, provided the universal quantifier is the leftmost one, as above, but not both simultaneously.

\begin{remark}\label{rem:divcongbasics}The following properties hold.
  \begin{enumerate}
  \item The relation $\tmid$ is a preorder.
  \item The relation $\smid$ is transitive, but not reflexive (see later, or~\cite[Lemma~6.4]{sobot_congruence_2021}).
  \item We have $u\equiv_w v$ if and only if  $w\tmid u\ominus v$.
  \item  We have $u\sequiv w v$ if and only if  $w\smid u\ominus v$.
  \item If $w=n\ne 0$ is principal, then both $\equiv_w$ and $\sequiv w$ coincide with the usual congruence relation modulo $n$.
\end{enumerate}
\end{remark}
From \Cref{defin:midequiv}, it is easy to obtain nonstandard characterisations; for example,  $u\sequiv w v$ if and only if whenever $d,a,b\in{}^{\ast}\mathbb{Z}$ are such that $d\models w, a\models u, b\models v$, then $d\mid  {}^{\ast}a-{}^{\ast\ast}b$. Below, we provide some further equivalent definitions of the divisibility relations.\footnote{For $\tmid$, this is in fact the original definition.} Denote by   $\mathcal U$ the family of all $\mid$-upward closed subsets of $\mathbb Z$.
\begin{remark}\label{rem:divchar}
For every $u,v\in \beta \mathbb Z$, the following hold.
\begin{enumerate}
\item We have  $u\tmid v$ if and only if $u\cap \mathcal U\subseteq v$.
\item We have  $u\smid v$ if and only if $\set{n\in \mathbb Z: n \mathbb Z\in v}\in u$.
\end{enumerate}
\end{remark}

\begin{fact}[{\cite[Lemma~5.6 and Theorem~5.7]{sobot_congruence_2021}}]\label{fact:sequiveqrel}
  For every $w\in \beta \mathbb Z \setminus \set 0$, the relation $\sequiv w$ is an equivalence relation compatible with $\oplus$ and $\odot$.
\end{fact}
An important role in our analysis of these notions will be played by those ultrafilters which are maximal with respect to divisibility amongst nonzero ultrafilters.
\begin{definition}
We denote by $\mathrm{MAX}$ the set of ultrafilters that are $\tmid$-divisible by all elements of $\beta\Z\setminus\{0\}$.
\end{definition}

The following characterisation may be proven by taking suitable tensor products. See also \cite[Lemma~5.8(a)]{sobot_congruence_2021}.

\begin{fact}\label{fact:maxsdiv}
The following are equivalent for $w\in \beta \mathbb Z \setminus \set 0$.
\begin{enumerate}
\item For every $u\in \beta \mathbb Z \setminus \set 0$ we have $u\tmid w$ (that is, $w\in \mathrm{MAX}$).
\item\label{point:maxaresmax} For every $u\in \beta \mathbb Z \setminus \set 0$ we have $u\smid w$.
\item For every $n\in \mathbb N$ we have $n\tmid w$ (that is, $w\equiv_n 0$, or equivalently $n \mathbb Z\in w$).
\end{enumerate}
\end{fact}
In the case of $\beta \mathbb N$, the following is~\cite[Lemma~4.3]{sobot_divisibility_2022}. Its version for $\beta \mathbb Z$ is proven in the same way. Recall that $K(\beta \mathbb Z, \odot)$ denotes the smallest bilateral ideal of the semigroup $(\beta \mathbb Z, \odot)$, and $\overline{K(\beta \mathbb Z, \odot)}$ denotes its closure.
\begin{fact}\label{max:ideal}
The set $\mathrm{MAX}$ is topologically closed in $\beta\Z$. Moreover, it is a $\odot$-bilateral ideal and it is closed under $\oplus$. In particular, $\overline{K(\beta\Z, \odot)}\subseteq\mathrm{MAX}$.
\end{fact}

Throughout, an important role will be played by the \emph{profinite completion} $\varprojlim \mathbb Z/n\mathbb Z$ of $(\mathbb Z, +)$, which may be thought of as the additive group of consistent choices of remainder classes modulo each $n\in \mathbb N$, and is usually denoted by $\hat{\mathbb Z}$. Explicitly, we may identify an element of $\hat{\mathbb Z}$ with a sequence $(a_n)_{n\in \mathbb N}$ such that  $a_n\in\set{0,\ldots, n-1}$ and, if $n\mid m$, then $a_m\equiv_n a_n$, with pointwise addition modulo $n$.  There is an isomorphism  $(\hat {\mathbb Z},+)\cong\prod_{p\in \mathbb P} (\mathbb Z_p,+)$, where $\mathbb Z_p$ denotes the $p$-adic integers. Again up to isomorphism, we may view $\hat{\mathbb Z}$ as the quotient of $({}^{\ast} \mathbb Z,+)$ by the equivalence relation that identifies $a$ and $b$ whenever, for every $n\in \mathbb N$, we have $a\equiv_n b$.

It is well-known that $\hat {\mathbb Z}$ is a \emph{profinite group}, that is, a topological group which is a Stone space, when equipped with the group topology where a basis of neighbourhoods of the identity is given by the clopen subgroups $n \hat{\mathbb Z}$. In other words, the  basic (cl)open sets are given by fixing finitely many remainder classes. The isomorphism $\hat{\mathbb Z}\cong \prod_{p\in \mathbb P} \mathbb Z_p$ is in fact an isomorphism of topological groups, that can be used to obtain a nice characterisation of the closed subgroups of $(\hat {\mathbb Z}, +)$. This may be proven directly, but it also follows from e.g.~Theorem~1.2.5 in \cite{wilson}, to which we also refer the reader interested in an introduction to profinite groups. Below, we adopt the convention that, if $\alpha$ is an infinite ordinal, then  $p^\alpha \mathbb Z_p=\set 0$.

\begin{fact}\label{fact:sbgrpsofzhat}
  View $\hat{\mathbb Z}$ as $\prod_{p\in \mathbb P} \mathbb Z_p$. Then, 
  the closed subgroups of $(\hat {\mathbb Z}, +)$ are precisely those of the form $\prod_{p\in \mathbb P} p^{\phi(p)}\mathbb Z_p$, where $\phi\from \mathbb P\to \omega+1$.

In particular, each closed subgroup may be written as $\set{x\in \hat {\mathbb Z}: \forall n\in D\; n \mid x}$, where $D$ is a $\mid$-downward-closed subset of $\mathbb Z$ of the form $\bigcap_{p\in \mathbb P}\left(p^{\phi(p)+1}\mathbb Z\right)^\cc$.
\end{fact}

\section{Congruences that are not equivalences}\label{sec:negative}

We begin this section by proving that $\equiv_w$ is not always an equivalence relation, thereby answering negatively a question of \v Sobot.
\begin{example}\label{ex1}
  Let $w\in \beta \mathbb Z\setminus \mathbb Z$ be such that, for every $n\in \mathbb N$, we have $w\equiv_n 1$. Then $\equiv_w$ is not transitive.
\end{example}
\begin{proof}
 For every $w$ we have  $w\tmid (-w)$, hence $0\equiv_w w$. On the other hand, for every $w$ as above,  $w\ominus 1\in \mathrm{MAX}$, so $w\equiv_w 1$, and by transitivity $0\equiv_w 1$, contradicting that $w$ is non-principal.
\end{proof}
Since any $w$ of the form $u\oplus 1$, with $u\in \mathrm{MAX}$ nonzero, satisfies the assumptions of \Cref{ex1}, this settles \cite[Question~7.1]{sobot_congruence_2021}. In the rest of this section, we study in more detail \emph{how} $\equiv_w$ may fail to be an equivalence relation. We easily observe that reflexivity is always guaranteed:
\begin{proposition}\label{fact:equivrefl}
For all $w\in \beta \mathbb Z\setminus\set 0$, the relation  $\equiv_w$ is reflexive.
\end{proposition}
\begin{proof}
Given any $u\in\beta\Z$, observe that $(a, a')\models u\otimes u$ implies that, for every $n\in\N$, we have $a\equiv_n a'$, hence that $n\Z\in u\ominus u$. By \Cref{fact:maxsdiv} $u\ominus u\in\mathrm{MAX}$, and we conclude by \Cref{rem:divcongbasics}.
\end{proof}

As observed in Example \ref{ex1}, transitivity of $\equiv_{w}$ is not guaranteed in general. Remarkably, failure of transitivity is the only obstruction to $\equiv_w$ being an equivalence relation.
\begin{theorem}\label{thm:transimpsym}
For every $w\in \beta \mathbb Z\setminus \set 0$, the relation $\equiv_w$ is an equivalence relation if and only if it is transitive.
\end{theorem}
\begin{proof}
By \Cref{fact:equivrefl}  we only need to show that if $\equiv_w$ is transitive, then it is symmetric.
Assume symmetry fails, as witnessed by $u,v$ such that  $u\equiv_w v$ but $v\not\equiv_w u$. Let $t\coloneqq u\ominus v$ and $t'\coloneqq -v\oplus u$. By construction $t\equiv_w 0$ and $t'\not\equiv_w 0$. On the other hand $t'\ominus t$ is easily checked to be in $\mathrm{MAX}$, hence $t'\equiv_w t$. It follows that $t'\equiv_w t\equiv_w 0$, but $t'\not\equiv_w 0$, so transitivity fails.
\end{proof}
Symmetry of $\equiv_{w}$ can also fail for reasons that have little to do with transitivity.\footnote{At least \emph{prima facie}, since we do not know whether a symmetric $\equiv_w$ must be automatically transitive, see \Cref{prob:symntrans}.} We will prove this by using upper Banach density, denoted by $\operatorname{BD}$.  For definitions and basic properties around densities, see e.g.~\cite{moreira_proof_2019} and references therein. Specifically, we will use the following consequence of~\cite[Theorem~2.1]{bergelson1985sets}; see~\cite[Corollary~2.4]{moreira_proof_2019} for more details. 
\begin{fact}
\label{dens:case}
Suppose that $\{B_n\}_{n\in\N}$ is a family of subsets of $\Z$, that $\Phi$ 
is a sequence of intervals of increasing length, and denote by $d_\Phi$ the associated density. If there is $\varepsilon>0$ such that, for every $n\in\N$, the density $d_\Phi(B_n)$ exists and  is larger than $\varepsilon$, then there is an infinite $X\subseteq\N$ such that the family $\{B_x: x\in X\}$ can be extended to a nonprincipal ultrafilter.
\end{fact}

\begin{theorem}
\label{not:sym}
Let $A\subseteq\Z$ be such that $\BD(A)>0$ and $A^\cc$ is thick. Then there are $u, v\in\beta\Z\setminus\Z$ such that $A\in u\oplus v$ and $A^\cc\in v\oplus u$.
\end{theorem}
\begin{proof}
Recall that a subset of $\mathbb Z$ is \emph{thick} iff it contains arbitrarily long intervals.  Hence, by assumption we can find, for every $n\in \mathbb N$, an interval $J_n=[a_n, b_n]$ such that $\abs{J_n}=2n+1$ and $J_n\subset A^\cc$. Denote by $\Phi$  the sequence of intervals along which $\BD(A)=d_\Phi(A)$ and observe that, for every $m\in\Z$, we have $d_\Phi(A-m)=d_\Phi(A)$.

Set $c_n\coloneqq (a_n+b_n)/2$ and apply \Cref{dens:case} with  $B_n\coloneqq A-c_n$, finding $Y\subseteq\{c_n\}_{n\in\N}\subseteq A^\cc$ such that $\{A-y: y\in Y\}$ is contained in a nonprincipal ultrafilter $v$.

  Fix any nonprincipal ultrafilter $u$ containing $Y$. By construction, $A\in u\oplus v$, and we are left to show that $A^\cc \in v \oplus u$.
Because $Y\subseteq \{c_n\}_{n\in\N}$, for every $a\in {}^{\ast}Y\setminus Y$, hence in particular for every $a\models u$,   and for every $n\in\Z$ we have  $n+a\in{}^{\ast}A^\cc$. Therefore, if $(b, a)\models v\otimes u$, we have  $b+a\in{}^{\ast}A^\cc$, concluding the proof.
\end{proof}

\begin{corollary}
\label{fail:sym}
There is $w\in \beta \mathbb Z\setminus\set 0$ such that $\equiv_w$ is not symmetric.
\end{corollary}
\begin{proof}
It is well-known that the set of squarefree integers has positive density, see e.g.~\cite{jakimczuk_simple_2013} for a short proof.
Moreover, an easy application of the Chinese Remainder Theorem shows that its complement is thick: if $p_k$ is the $k$-th prime, it suffices to find an integer $n$ such that $n\equiv_{p_k^2}-k$ for sufficiently many $k$. By \Cref{not:sym}, and the fact that squarefree integers form a symmetric subset of $\mathbb Z$, there exist $u, v\in\beta\Z\setminus\Z$ such that $u\ominus v$ is squarefree (that is, it contains the set of squarefree integers; equivalently, its realisations are squarefree) and $v\ominus u$ is not. Since $v\ominus u$ is not squarefree, it is divided by some square $w>1$, which cannot divide the squarefree $u\ominus v$.
\end{proof}
\begin{remark}\label{rem:failsym}
The proof of Corollary \ref{fail:sym} also works if we fix  an arbitrary $\alpha\from \mathbb{P}\to\omega+1$  such that $\alpha(p)>1$ for every $p\in \mathbb{P}$ and (recalling that we convene $p^{\omega}\Z=\set 0$), replace the squarefree integers by  $A\coloneqq\bigl(\bigcup_{p\in\mathbb{P}}p^{\alpha(p)}\Z\bigr)^\cc$,  the squarefree case corresponding to $\alpha$ being constantly $2$. The set $A$ has positive density because it contains the squarefree integers, and its complement is again proven to be thick by using the Chinese Remainder Theorem.
\end{remark}

\section{Self-divisible ultrafilters}\label{sec:main}

In the previous section, we have seen examples of ultrafilters $w$ such that $\equiv_{w}$ is not an equivalence relation, namely all those in $\mathrm{MAX}\oplus 1$ except $1$. On the other hand, there are ultrafilters $w$ such that $\equiv_w$ is an equivalence relation, for instance all principal ones.\footnote{Keep reading for less trivial examples.} It is natural to look for a characterisation of when this happens; in this section, we provide a complete solution to this problem.

\begin{definition} Let $w\in\beta\Z$.
  \begin{enumerate}
  \item We denote the set of integers dividing $w$ by $D(w)\coloneqq\set{n\in \mathbb Z: n \mathbb Z\in w}$.
  \item We call $w\in \beta \mathbb Z\setminus \set 0$ \emph{self-divisible} iff $D(w)\in w$.
\end{enumerate}
\end{definition}

\begin{remark}\label{rem:dfsmid}
  By \Cref{rem:divchar}, $w\smid u$ if and only if  $D(u)\in w$ if and only if for some (equivalently, every)  $(a,b)\models w\otimes u$ we have $a\mid b$. In particular, $w$ is self-divisible if and only if $w\smid w$. In nonstandard terms, this means that, whenever $a\in {}^{\ast} \mathbb Z$ generates $w$, we have $a\mid {}^{\ast} a$.
\end{remark}
As anticipated in the introduction, the relations $\equiv_w$ and $\sequiv w$ have complementary shortcomings: the former is not, in general, an equivalence relation; as for the latter, there are ultrafilters $w$ such that $w\nsequiv w0$, equivalently, such that $w \nsmid w$, for example any $w$ which is divided by no $n>1$. By definition, the self-divisible ultrafilters are those such that $\sequiv w$ is well-behaved in this respect. Perhaps unexpectedly, we will prove below that the   self-divisible $w$ are also precisely those for which the weak congruence $\equiv_w$ is well-behaved, that is, is an equivalence relation.

We will do this via a small detour in the realm of profinite integers $\hat{\mathbb Z}$. This is no coincidence, since we will later show in \Cref{rem:vvprime} that $(\beta\mathbb Z, \oplus)/\mathord{\sequiv w}$ is isomorphic to a quotient of $\hat {\mathbb Z}$. The connection between $\beta \mathbb Z$ and $\hat{\mathbb Z}$ arises naturally from the fact that every $u\in \beta \mathbb Z$ induces a consistent choice of remainder classes modulo the standard integers, that is, an element of $\hat {\mathbb Z}$. Let us give a name to the corresponding function.
\begin{definition}
Define $\pi\from \beta \mathbb Z \to \hat{\mathbb Z}$ as the map sending each ultrafilter to the sequence of its remainder classes.  
\end{definition}
Our detour will lead us to talk about the following sets.
\begin{definition}
  We denote by $Z_w\coloneqq\set{u\in \beta \mathbb Z: w\tmid u}$ the set of ultrafilters divisible by $w$.
\end{definition}
\begin{remark}
By \Cref{rem:divchar}  the set $Z_w$ is a closed subset of $\beta \mathbb Z$, corresponding to the filter of $\mid$-upward-closed elements of $w$.
\end{remark}
The remark above has a converse. Since we will never use it,  we leave the (standard) proof to the reader.
\begin{remark}
If $\mathcal F\subseteq \mathcal U$ is a family of $\mid$-upward-closed subsets of $\mathbb Z$, then there is $w\in \beta \mathbb Z$ such that $\mathcal F=w\cap \mathcal U$ if and only if $\mathcal F$ is a prime filter on the distributive lattice $\mathcal U$.
\end{remark}
\begin{lemma}\label{lemma:pitoZhat}The following statements hold.
  \begin{enumerate}
  \item The map $\pi$ is continuous, surjective, and a homomorphism of semigroups
    $(\beta \mathbb Z, \oplus)\to (\hat{\mathbb Z}, +)$.  
  \item The quotient topology induced by $\pi$ coincides with the usual topology of $\hat{\mathbb Z}$.
  \item If $Z_w$ is
    closed under $\oplus$, then $\pi(Z_w)$ is a closed subgroup of $(\hat {\mathbb Z}, +)$.
\end{enumerate}
\end{lemma}
\begin{proof}
The first point is an easy exercise, and the second one follows from the facts that every continuous surjective map from a compact space to a Hausdorff one is closed, and that every continuous, surjective, closed map is automatically a quotient map. As for the last point, (for every $w$) the set $Z_w$ contains $0$ and is closed under $-$. Grouphood of $\pi(Z_w)$ then follows easily from the assumption, and we conclude by observing that  both spaces are compact Hausdorff.
\end{proof}

\begin{lemma}\label{lemma:sdoop}
For every $w\in \beta \mathbb Z \setminus \set 0$ and every $u,u^{\prime}\in\beta\Z$, if $\pi(u)=\pi(u')$, then $u\sequiv w u'$. 
\end{lemma}
\begin{proof}
If $\pi(u)=\pi(u')$, then $u\ominus u'\in \mathrm{MAX}$, and we conclude by \Cref{fact:maxsdiv} and \Cref{rem:divcongbasics}.
\end{proof}

\begin{lemma}\label{lemma:equivwcongr}
  If $\equiv_w$ is an equivalence relation, then it is a congruence with respect to $\oplus$.
\end{lemma}
\begin{proof}
  Assume that $u\equiv_w v$, and let us show that for all $t\in\beta\Z$ we have  $u\oplus t\equiv_w v\oplus t$.   By definition, $u\ominus v\equiv_w 0$, and it is easy to see that $\pi(u\ominus v)=\pi(u\oplus t \ominus v \ominus t)$.  By \Cref{lemma:sdoop} we have $u\ominus v\sequiv w u\oplus t \ominus v \ominus t$. Whenever $\equiv_w$ is an equivalence relation, it is automatically a coarser one than $\sequiv w$ by definition, so $u\oplus t \ominus v \ominus t\equiv_w 0$, hence  $u\oplus t \equiv_w v\oplus t$. The proof that  $t\oplus u \equiv_w t\oplus v$  is analogous.
\end{proof}
We are now ready to prove the first part of our main result. Later, in \Cref{thm:charred}, we will see several more properties equivalent to self-divisibility.
\begin{theorem}\label{main:thm}
  For every $w\in \beta \mathbb Z \setminus \set 0$, the following are equivalent.
  \begin{enumerate}
  \item \label{point:mainthmwsmidw} The ultrafilter $w$ is self-divisible.
  \item \label{point:mainthmees} The relations $\equiv_w$ and $\sequiv w$ coincide.
  \item \label{point:mainthmeqrel} The relation $\equiv_w$ is an equivalence relation.
  \item \label{point:mainthmphi} For every $u$, we have  $w\tmid u$ if and only if  $D(w)\subseteq D(u)$.
  \item \label{point:mainthmgcd} For every $a,b\models w$  there is  $c\models w$ such that $c\mid \gcd(a,b)$.
  \end{enumerate}
\end{theorem}
\begin{proof} In order to prove $\eqref{point:mainthmwsmidw}\IMP\eqref{point:mainthmees}$, we need to show that if $u\equiv_w v$ then $u\sequiv w v$, since the converse is always true. Let $d\models w$ and $(a,b)\models u\otimes v$ be such that $d\mid a-b$. Let $t\coloneqq \tp(d,a,b/\mathbb Z)$ and find, using saturation, some $d'$ such that $(d', (d,a,b))\models w \otimes t$. In particular, $(d', (a,b))$ is a tensor pair, hence by  associativity of $\otimes$, we have $(d',a,b)\models w\otimes u\otimes v$, therefore  $u\sequiv w v$ if and only if $d'\mid a-b$. But $(d', d)\models w\otimes w$, hence $d'\mid d$ by assumption and \Cref{rem:dfsmid}.

  The implication  $\eqref{point:mainthmees}\IMP \eqref{point:mainthmeqrel}$ is obvious because $\sequiv w$ is an equivalence relation by \Cref{fact:sequiveqrel}.

To prove $\eqref{point:mainthmeqrel}\IMP\eqref{point:mainthmphi}$,  assume that $\equiv_w$ is an equivalence relation. It is easy to see, for instance by using \Cref{rem:divchar}, that $w\tmid u\IMP D(w)\subseteq D(u)$  always holds,  so let us assume that $D(w)\subseteq D(u)$. We need to prove that $w\tmid u$ or, in other words, that $u\in Z_w$.
  We begin by observing that, for every $v$, $v'$, if $\pi(v)=\pi(v')$, then $v\sequiv w v'$ by \Cref{lemma:sdoop}, hence, since we are assuming $\equiv_w$ is an equivalence relation, $v\in Z_w\IFF v'\in Z_w$, so it suffices to show that $\pi(u)\in \pi(Z_w)$.   By \Cref{lemma:equivwcongr}, $\equiv_w$ is a congruence with respect to $\oplus$, thus $Z_w$ is closed under $\oplus$. By \Cref{lemma:pitoZhat}, $\pi(Z_w)$ is a closed subgroup of $(\hat{\mathbb Z}, +)$, therefore, by \Cref{fact:sbgrpsofzhat}, there is a $\mid$-downward closed $D\subseteq \mathbb Z$ such that $\pi(Z_w)=\set{x\in \hat {\mathbb Z}: \forall n\in D\; n \mid x}$. In other words, for every $v$, we have $\pi(v)\in \pi(Z_w)\IFF D\subseteq D(v)$. Trivially,  $\pi(w)\in \pi(Z_w)$, hence $D\subseteq D(w)$, and we conclude by using the assumption that $D(w)\subseteq D(u)$.

    For $\eqref{point:mainthmphi}\IMP\eqref{point:mainthmgcd}$, observe that if $n\tmid w$ then automatically $n\mid \gcd(a,b)$. Hence  $D(w)\subseteq D(\tp(\gcd(a,b)/\mathbb Z))$, and the conclusion follows.

    To prove $\eqref{point:mainthmgcd}\IMP\eqref{point:mainthmwsmidw}$, let $(a,b)\models w\otimes w$. By assumption, there is $c\models w$ such that $c\mid \gcd(a,b)$. This implies that $\abs c\le \abs a$, hence by \Cref{fact:abstensorpair} we have  $(c,b)\models w\otimes w$, witnessing self-divisibility. \end{proof}

\section{Examples}\label{sec:eg}
In this section we look at some examples and non-examples of self-divisible ultrafilters and see how this notion interacts with other fundamental classes of ultrafilters, such the idempotent or the minimal elements of the semigroups $(\beta \mathbb N, \oplus)$ and $(\beta \mathbb N, \odot)$.
We also define a special kind of self-divisible ultrafilters, the \emph{division-linear} ones, and look at the relation between the shape of $D(w)$ and self-divisibility of $w$.
\begin{example}\label{eg:longlist}\*
  \begin{enumerate}
  \item Clearly, every nonzero principal ultrafilter is self-divisible.
  \item \label{point:idempmin}Every ultrafilter in $\mathrm{MAX}$ is easily checked to be self-divisible. 
  \item Every  ultrafilter of the form $\tp(a/\mathbb Z)$, where $a> \mathbb Z$ is prime, is not self-divisible.
      \item   If $p_1,\ldots, p_n\in \mathbb P$, and $a_1,\ldots, a_n, b\in {}^{\ast} \mathbb N$, then  $\tp(p_1^{a_1}\cdot\ldots\cdot p_n^{a_n}\cdot b/\mathbb Z)$ is self-divisible if and only if $\tp(b/\mathbb Z)$ is.  This can be easily seen by using point~\ref{point:mainthmgcd} of \Cref{main:thm}.
      \item In particular, every ultrafilter of the form $\tp(p_1^{a_1}\cdot\ldots\cdot p_n^{a_n}/\mathbb Z)$ is self-divisible.
          \item  Self-divisible ultrafilters form a semigroup with respect to $\odot$.
        \item\label{point:odotminimal}
 By~\eqref{point:idempmin} above and \Cref{max:ideal}, all $\odot$-minimal ultrafilters are self-divisible.
      \item  By~\eqref{point:idempmin} above,  all $\oplus$-idempotent
    ultrafilters are self-divisible.
\item \label{point:minnotsd}  If $u\ne 0$ is a minimal $\oplus$-idempotent, then $u\oplus 1$ is $\oplus$-minimal, but not self-divisible, since $D(u\oplus 1)=\set{1, -1}$.
\end{enumerate}
\end{example}
We will see in \Cref{eg:odotidemp} that, in~\ref{point:minnotsd} above, the reverse inclusion does not hold either. In fact, $\oplus$-minimality is not even implied by the following  strengthening of self-divisibility.
\begin{definition}
An ultrafilter $w\in\beta\Z$ is \emph{division-linear} iff $w\ne 0$ and for every pair $d,d'$ of realisations of $w$ we have that $d\mid d'$ if and only if $\abs{d}\leq\abs{d'}$.
\end{definition}
\begin{remark}
By point~\ref{point:mainthmgcd} of \Cref{main:thm}, every division-linear ultrafilter is self-divisible.
\end{remark}
We already said that every nonzero principal ultrafilter is self-divisible; in fact, every such ultrafilter is division-linear. We now look at another example, then characterise division-linearity. For a similar characterisation of self-divisibility, see \Cref{thm:charred}.
\begin{example}\label{eg:odotidemp}
Every ultrafilter $u$ containing the set $F\coloneqq\{n!: n\in\N\}$ of  factorials is  division-linear. Because $F$ is not a  multiplicative IP set,  $u$ is not $\odot$-idempotent and, because $F$ is not piecewise syndetic, $u$ is not in $\overline{K(\beta\Z, \oplus)}$.
\end{example}
\begin{proposition}\label{pr:autodivlinord}
A nonzero $w\in \beta \mathbb Z$ is division-linear if and only if it contains a set linearly preordered by divisibility.   
\end{proposition}
\begin{proof}
   If $A\in w$ is linearly preordered by divisibility, then the conclusion follows by observing that whenever $d,d'\models w$ then $d,d'\in {}^{\ast}A$. Conversely, if we think of $w$ as a type, then $w$ is division-linear, by definition, if and only if $w(x)\cup w(y) \vdash (\abs x\le \abs y) \to (x\mid y)$. By compactness, there are $A, B\in w$ such that $(x\in {}^{\ast}A) \land (y\in {}^{\ast}B) \vdash(\abs x\le \abs y) \to (x\mid y)$. It follows that $A\cap B$ is linearly preordered by divisibility.
\end{proof}

\begin{proposition}
There is an ultrafilter in $\mathrm{MAX}$ (in particular, a self-divisible ultrafilter) which is not division-linear.
\end{proposition}
\begin{proof}
  Let $\mathcal L$ be the family of subsets of $\mathbb Z\setminus \set 0$ which are linearly preordered by divisibility. By \Cref{pr:autodivlinord}, it suffices to prove that the family $\mathcal F\coloneqq\set{n \mathbb Z : n>1}\cup \set{L^\cc : L\in \mathcal L}$ has the finite intersection property.  But this is clear, because $\set{n \mathbb Z : n>1}$ is closed under finite intersections and every $n \mathbb Z$ contains an infinite $\mid$-antichain, hence cannot be contained in a finite union of elements of $\mathcal L$. 
\end{proof}

The reverse inclusion also fails:

\begin{example}\label{eg:pow2}
If $u$ is the type of a nonstandard power of $2$, then it is division-linear, hence self-divisible, but not in $\mathrm{MAX}$.
\end{example}
\begin{example}
Self-divisibility is not preserved upwards nor downwards by $\tmid$. For instance, if  $v\in \mathrm{MAX}$, then for any non-self-divisible $w$ we have both  $w\tmid v$ and $w\smid v$ (\Cref{fact:maxsdiv}). In the other direction, fix infinite $a,b$ with $b$ prime and take $w\coloneqq\tp(2^a/\mathbb Z)$ and $v\coloneqq\tp(2^ab/\mathbb Z)$. It is easy to show that $w(x)\otimes v(y)\vdash x\mid y$, from which we deduce $w\smid v$, and in particular $w\tmid v$. By \Cref{eg:pow2} $w$ is division-linear, hence self-divisible, but $v$ is not. 
\end{example}
We saw in point~\eqref{point:odotminimal} of \Cref{eg:longlist} and in \Cref{eg:odotidemp} that $\odot$-minimal ultrafilters are self-divisible, and that division-linearity does not imply $\odot$-idempotency. Moreover, it is easily seen that every nonprincipal ultrafilter containing the set of primes is neither self-divisible, nor $\odot$-idempotent.
 \Cref{idemp:notdf}  and \Cref{co:nolinkodot} below complete the picture.
\begin{proposition}
  \label{idemp:notdf}  
There exist $\odot$-idempotent non-self-divisible ultrafilters.
\end{proposition}
\begin{proof}
Recall that an ultrafilter is \emph{$\N$-free} iff it is not divisible by any $n>1$, see~\cite{sobot_congruence_2021} and references therein. Denote the set of $\N$-free ultrafilters by $\mathrm{Free}\coloneqq\bigcap_{n>1}\overline{n\Z^\cc}$. Let us show that this set is closed under $\odot$. Indeed, $w\in\mathrm{Free}$ if and only if for every $a\models w$ and for every $n>1$  we have $n\nmid a$. If $(a, b)\models w\otimes v$ and $n\mid a\cdot b$, then every prime divisor of $n$ must divide either $a$ or $b$, which is a contradiction because $n>1$. From this it follows easily that $\mathrm{Free}\setminus \set{1, -1}$ is  a closed subset of $\beta\Z$ closed under $\odot$.  It is therefore a compact right topological semigroup, and by Ellis' Lemma it must contain a $\odot$-idempotent. To conclude, note that $\mathrm{Free}$ does not contain any nonprincipal self-divisible ultrafilters, since every $w\in\mathrm{Free}$ has $D(w)=\set{1,-1}$.
\end{proof}
\begin{lemma}\label{lemma:dlnoapl3}
  If $A\subseteq \mathbb N$ is linearly ordered by divisibility, then $A$ contains no arithmetic progression of length $3$.
\end{lemma}
\begin{proof}
A counterexample $a,a+b, a+2b\in A$ should satisfy $(a+b)\mid (a+b)+b$, hence $(a+b)\mid b$, a contradiction.
\end{proof}
The following fact is well-know,  easy to prove, and a special case of the much more general~\cite[Example~5.6]{LMON}.
\begin{fact}\label{fact:apl3ideal}
The set of ultrafilters $u$ such that every element of $u$ contains an arithmetic progression of length $3$  is a closed bilateral ideal of $(\beta \mathbb N, \odot)$.
\end{fact}
\begin{corollary}\label{co:nolinkodot}
  There is no nonzero, division-linear ultrafilter in $\overline{K(\beta\N, \odot)}$.
\end{corollary}
We thank the referee for catching a mistake in a previous version of the proof below.
\begin{proof}
If   $u\in \beta \mathbb N$  is a counterexample, by \Cref{pr:autodivlinord} and the fact that we are working over $\mathbb N$, some $A\in u$  is linearly ordered by divisibility, hence by \Cref{lemma:dlnoapl3} it contains no arithmetic progression of length $3$. This contradicts $u\in \overline{K(\beta\N, \odot)}$ by \Cref{fact:apl3ideal}.
\end{proof}

 Self-divisible ultrafilters are not closed under $\oplus$: it suffices to sum any $u\in \mathrm{MAX}$ with $v=1$. In fact, it is possible to construct a counterexample with both $u,v$ nonprincipal.
      \begin{proposition}
        There are division-linear, nonprincipal ultrafilters $u,v$ such that $u\oplus v$ is not self-divisible.
      \end{proposition}
      \begin{proof}
Let $f\from \mathbb N\to \mathbb P$ be the increasing enumeration of all primes, and fix $a\in {}^{\ast}\mathbb Z$ such that $a> \mathbb Z$. Let $u\coloneqq \tp(\prod_{0<b\le a}{}^{\ast}f(2b)/\mathbb Z)$ and  $v\coloneqq \tp(\prod_{0<b\le a}{}^{\ast}f(2b-1)/\mathbb Z)$. It is easy to see that $u$, $v$ are division-linear. Because every $p\in \mathbb P$ divides precisely one between $u$ and $v$, we have $D(u\oplus v)=\set{1,-1}$. Since $u\oplus v$  is nonprincipal, it cannot be self-divisible.
      \end{proof}

Point \ref{point:mainthmphi} of \Cref{main:thm} might suggest that self-divisibility of $w$ could be deduced just by looking at $D(w)$. Rather surprisingly, this is false, except in some trivial cases. To prove this, it will be convenient to replace $D(w)$ by the following function on the set  $\mathbb{P}$ of prime natural numbers.  The reader familiar with algebra may recognise that such functions are exactly the same as \emph{supernatural numbers} in the sense of Steinitz.
\begin{definition}
    Given $u\in \beta\Z$, we define $\varphi_u\from \mathbb{P}\to \omega+1$  as the function sending $p$ to $\max\{k\in \omega: p^k\Z\in u\}$ if this exists, and to $\omega$ otherwise.
  \end{definition}
  \begin{remark}\*
    \begin{enumerate}
  \item By definition,
for every $u\in \beta\Z$, the set $D(u)$ determines $\varphi_u$, and conversely. More explicitly, recalling our convention that  $p^{\omega+1}\Z=\set 0$, we have
\[
  D(u)=\bigcap_{p\in \mathbb{P}}(p^{\varphi_u(p)+1}\Z)^\cc.
\]
\item  The set 
\begin{equation}\label{eq:duc}\tag{$\dagger$}
  D(u)^\cc=\{k: k\Z \not\in u\}=\bigcup_{p\in \mathbb{P}}p^{\varphi_u(p)+1}\Z
\end{equation}
is $\mid$-upward closed.
  \end{enumerate}
\end{remark}

\begin{definition} Let $\varphi\from \mathbb{P}\to \omega+1$

    We say that $\varphi$ is \emph{finite} iff $\varphi\inverse(\set{\omega})=\emptyset$ and $\varphi\inverse(\mathbb N)$ is finite.

    We say that $\varphi$ is \emph{cofinite} iff $\varphi\inverse(\set 0)\cup\varphi\inverse(\mathbb N)$ is finite, that is, $\varphi\inverse(\set{\omega})$ is cofinite.
\end{definition}
Intuitively, $\phi_u$ is finite whenever $u$ is only divisible by finitely many integers, and $\phi_u$ is cofinite whenever, for all but finitely many $p\in \mathbb P$, the ultrafilter $u$ is divisible by every power of $p$.
\begin{proposition}\label{pr:phidf}Let $\varphi\from \mathbb{P}\to\omega+1$.
    \begin{enumerate}
        \item If $\varphi$ is finite, then every $w\in \beta\Z$ such that $\varphi_w=\varphi$ is either principal or not self-divisible.
        \item If $\varphi$ is cofinite, then every $w\in \beta\Z$ such that $\varphi_w=\varphi$ is self-divisible.
        \item In every other case, namely whenever $\varphi$ is not finite neither cofinite, there exist $u, v\in \beta\Z$ such that $\varphi_u=\varphi_v=\varphi$ with $u$ self-divisible and $v$ not self-divisible.
    \end{enumerate}
\end{proposition}
\begin{proof} The first point is immediate from the fact that, if $\varphi_w$ is finite, then $D(w)$ is finite. As for the second point, if $\varphi_w$ is cofinite then, by definition, for every $p$ outside of a certain finite $P_0\subseteq \mathbb P$ we have  $p^{\varphi(p)+1}\Z=\set 0$. In other words, the union in~$\eqref{eq:duc}$ is actually a finite union, namely  \[D(w)^\cc=\bigcup_{p\in P_0}p^{\varphi(p)+1}\Z.\] By definition of $\varphi$, this is a finite union of sets not in $w$, hence does not belong to $w$.

    Towards the last point, define
  \begin{align*}
    \mathcal{F}&\coloneqq \{p^{\varphi(p)}\Z\cap (p^{\varphi(p)+1}\Z)^\cc:p\in \varphi\inverse(\mathbb N)\}\\
               &\cup\{(p\Z)^\cc: p\in \varphi\inverse(\set 0)\}\\
               &\cup\{p^k\Z: k\in \N, p\in \varphi\inverse(\set{\omega})\}.
  \end{align*}
 Every ultrafilter $w$ extending $\mathcal F$ will have $D(w)=(\bigcup_{p\in \mathbb{P}}p^{\varphi(p)+1}\Z)^\cc$, so we need to prove that, if $\varphi$ is not finite nor cofinite, then the families below have the finite intersection property
\[
\mathcal F\cup\set[\Big]{\bigcup_{p\in \mathbb{P}}p^{\varphi(p)+1}\Z};
\qquad \mathcal F \cup\set[\Big]{\Bigl(\bigcup_{p\in \mathbb{P}}p^{\varphi(p)+1}\Z\Bigr)^\cc}.
\]
If $\mathcal I\subseteq \mathcal F$ is finite, then its intersection may be written as follows, for some $p_1, \ldots, p_k\in\varphi\inverse(\mathbb N)$, some $q_1, \ldots, q_s\in\varphi\inverse(\set 0)$, some $r_1, \ldots, r_\ell\in \varphi\inverse(\set{\omega})$, and some $n_1, \ldots, n_\ell\in \N$:
  \begin{align*}
\bigcap    \mathcal{I}&= p_1^{\varphi(p_1)}\Z\cap\ldots\cap p_k^{\varphi(p_k)}\Z\cap (p_1^{\varphi(p_1)+1}\Z)^\cc\cap\ldots\cap (p_k^{\varphi(p_k)+1}\Z)^\cc\\
               &\cap(q_1\Z)^\cc\cap\ldots\cap(q_s\Z)^\cc\\
    &\cap r_1^{n_1}\Z\cap\ldots\cap r_\ell^{n_\ell}\Z.
  \end{align*}
  Define   \(
    a\coloneqq p_1^{\varphi(p_1)}\cdot\ldots\cdot p_k^{\varphi(p_k)}\cdot r_1^{n_1}\cdot\ldots\cdot r_\ell^{n_\ell}
  \)
and observe that it belongs to $\bigcap\mathcal I$. By definition, if  $\varphi\inverse(\mathbb N)$ is infinite, then automatically $\phi$ is neither finite nor cofinite. Infinity of $\varphi\inverse(\mathbb N)$ gives us some  $p_\dagger\in \varphi\inverse(\mathbb N)\setminus\{p_1, \ldots, p_k\}$, so we get that  $a\cdot p_\dagger^{\varphi(p_\dagger)+1}\in\bigcap\mathcal{I}\cap \bigcup_{p\in \mathbb{P}}p^{\varphi(p)+1}\Z$, and that
  $a\cdot p_\dagger\in\bigcap\mathcal{I}\cap \left(\bigcup_{p\in \mathbb{P}}p^{\varphi(p)+1}\Z\right)^\cc$.

If instead $\varphi\inverse(\mathbb N)$ is finite then, because we are assuming that $\phi$ is not cofinite, the set $\varphi\inverse(\set 0)$ is infinite. Let $a$ be as above. If  $q\in \varphi\inverse(\set 0)\setminus\{q_1, \ldots, q_s\}$, then $a\cdot q\in \bigcap\mathcal{I}\cap \bigcup_{p\in \mathbb{P}}p^{\varphi(p)+1}\Z$. Moreover, $a\in\mathcal{I}\cap\left(\bigcup_{p\in \mathbb{P}}p^{\varphi(p)+1}\Z\right)^\cc$, and we are done.
      \end{proof}
      \section{A bit of topology}\label{sec:topology}
In this section we study the topological properties of the subspaces of self-divisible and of division-linear ultrafilters. We begin with an easy remark.
      \begin{remark}
  For every $u\in \beta \mathbb Z$ we have the following.
  \begin{enumerate}
  \item \cite[Lemma~1.3]{sobot_divisibility_2020} The set $\set{w: w\tmid u}$ is closed: it coincides with \[\bigcap \set{\overline{B}: B\in u, B \text { is $\mid$-downward closed}}.\] This follows from \Cref{rem:divchar} and the fact that $A$ is $\mid$-upward closed if and only if $A^\cc$ is $\mid$-downward closed.
  \item The set $\set{w: w\smid u}$ is clopen: it coincides with $\overline{D(u)}$.
\end{enumerate}
\end{remark}
The set of division-linear ultrafilters contains all the nonzero principal ultrafilters, and it follows that its closure is $\beta \mathbb Z\setminus \set{0}$. Therefore, we look at the topological properties of self-divisible and division-linear ultrafilters in the subspace of nonprincipal ultrafilters.
\begin{definition}
Let $\DF$ denote the set of self-divisible nonprincipal ultrafilters.  Similarly, denote by $\DL$ the set of division-linear  nonprincipal ultrafilters. Let $\overline{\DF}, \overline{\DL}$ be their topological closures in $\beta\Z\setminus \mathbb Z$.
\end{definition}
\begin{proposition}\label{clos:df}We have 
\(
\overline{\DF}=\overline{\DL}=\{w\in\beta\Z: \forall A\in w\;\exists X\subseteq A\;( X\text{ is an infinite $\mid$-chain})\}.
\)
\end{proposition}
\begin{proof}
Let 
\(
\Psi(w)\coloneqq \forall A\in w\;\exists X\subseteq A\;( X\text{  is an infinite $\mid$-chain}).
\)
First of all, we show that $\Psi(w)$ holds for every $w\in\DF$. For such a $w$, by definition $D(w)\in w$. Fix $A\in w$ and let $x_1\in A\cap D(w)\in w$. Since $x_1\in D(w)$, we have  $x_1\Z\in w$. If we take $x_2\in x_1\Z\cap A\cap D(w)\in w$ such that $\abs {x_2}>\abs {x_1}$, then $x_2\Z\in w$, and by induction we obtain the desired infinite chain $X=\{x_1\mid x_2\mid \ldots\}\subseteq A$.

 If $\Psi(w)$ holds, and $A\in w$, then every nonprincipal $u$ containing an infinite linearly ordered $X\subseteq A$ must be division-linear. Therefore,  in every open neighbourhood of $w$ we can find an element of $\DL$, hence $\Psi(w)$ implies that $w\in\overline{\DL}$. Conversely, assume $w\in\overline{\DL}$. Then for every $A\in w$ there exists $u\in\DL$ such that $A\in u$. If $X\in u$ witnesses division-linearity of $u$, then $X\cap A\in u$ is a linearly ordered infinite subset of $A$, hence $\Psi(w)$ holds.

We conclude by observing that $\DL\subseteq\DF\subseteq\overline{\DL}$ implies $\overline{\DF}=\overline{\DL}$.
\end{proof}
We call \emph{additive} (\emph{multiplicative}, respectively) \emph{Hindman ultrafilters} those in the closure of the nonprincipal  $\oplus$-idempotents ($\odot$-idempotents, respectively).
Because $\oplus$-idempotents are in $\mathrm{MAX}$, which is topologically closed,  every additive Hindman ultrafilter belongs to $(\mathrm{MAX}\setminus\set 0)\subseteq\DF$.
\begin{corollary}
\label{mult:hind}
Every multiplicative Hindman ultrafilter is in $\overline{\DF}$.
\end{corollary}
\begin{proof}
An ultrafilter $w$ is multiplicatively Hindman if and only if for every $A\in w$ there exists an increasing sequence $(x_i)_{i\in\omega}$ of integers such that $\operatorname{FP}((x_i)_{i\in\omega})\coloneqq \{x_{i_0}\cdot\ldots\cdot x_{i_k}: i_0< \ldots< i_k\in \omega\}$ is a subset of $A$. Now, notice that $\{x_0\cdot\ldots\cdot x_k: k\in\omega\}$ is a proper subset of $\operatorname{FP}((x_i)_{i\in\omega})\subseteq A$ linearly ordered by divisibility. By \Cref{clos:df} we conclude that $w\in\overline{\DF}$.
\end{proof}
Combining \Cref{mult:hind}, \Cref{idemp:notdf}, and \Cref{clos:df}, we obtain the following. 
\begin{corollary}
  The sets $\DF$ and $\DL$ are not closed in $\beta \Z \setminus \Z$.
\end{corollary}
\begin{proof}
  Every ultrafilter provided by \Cref{idemp:notdf} lies in $\overline{\DF}\setminus \DF$ by \Cref{mult:hind}, and $\overline{\DF}\setminus \DF=\overline{\DL}\setminus\DF\subseteq\overline{\DL}\setminus\mathrm{DL}$.

\end{proof}

\section{A bit of (topological) algebra}\label{sec:algebra}
 We study the quotients $\beta \mathbb Z/\mathord{\sequiv w}$ and prove some additional characterisations of self-divisibility.

 Recall that every relation $\sequiv w$ is a congruence with respect to $\oplus$ (\Cref{fact:sequiveqrel})
\begin{remark}\label{rem:vvprime}\*
    \begin{enumerate}
 \item Since, by \Cref{lemma:sdoop}, the $\sequiv w$-class of $u\in \beta \mathbb Z$ only depends on its image in $\hat{ \mathbb Z}$, the quotient map $\rho_w\from \beta \mathbb Z\to \beta \mathbb Z/\mathord{\sequiv w}$ factors through a well-defined map $\sigma_w\from \hat{\mathbb  Z}\to \beta \mathbb Z/\mathord{\sequiv w}$, which sends $\pi(u)$ to $u/\mathord{\sequiv w}$. Note that $\sigma_w$ is a homomorphism of groups.
\item   Let $w\in \beta \mathbb Z$ be such that $w>1$, and view $\hat{\mathbb Z}$ as a subgroup of $\prod_{n\ge 2} \mathbb Z/{n \mathbb Z}$. Let $\hat{\mathbb Z}/w$  be the image of $\hat{\mathbb Z}$ under the projection from $\prod_{n\ge 2} \mathbb Z/{n \mathbb Z}$ onto the ultraproduct $\prod_w \mathbb Z/{n \mathbb Z}$.  The obvious map $\beta \mathbb Z/\mathord{\sequiv w}\to \hat {\mathbb Z}/w$ is (well-defined and) an isomorphism making the diagram in \Cref{figure:commdiag} commute.
  \begin{figure}[b]\caption{Diagram from \Cref{rem:vvprime}.}\label{figure:commdiag}
\begin{center}
\begin{tikzpicture}[scale=3]
\node(nw) at (0,0.5){$\beta \mathbb Z$};
\node(ne) at (1,0.5){$\hat {\mathbb Z}$};
\node (sw) at (0,0) {$\beta \mathbb Z/\mathord{\sequiv w}$};
\node(se) at (1,0){$\hat {\mathbb Z}/w$};

\path[->, thick,  font=\scriptsize]
(nw) edge node [above] {$\pi$} (ne)
(sw) edge node [above] {$\cong$} (se)
(nw) edge node [left] {$\rho_w$} (sw)
(ne) edge node [above] {$\sigma_w\;$} (sw)
(ne) edge node [right] {} (se);
\end{tikzpicture}
\end{center}
\end{figure}
 By using commutativity of the diagram, together with the fact that $\pi$ is a closed map, it is easy to check that $\sigma_w$ is continuous with respect to the quotient topology on $\beta \mathbb Z/\mathord{\sequiv w}$, and in fact induces the same quotient topology, that is, $\sigma_w\inverse(C)$ is closed if and only if $C$ is closed (if and only if $\rho_w\inverse(C)$ is closed).
\item 
  In particular,  if $w$ is self-divisible then, by \Cref{main:thm} and \Cref{lemma:pitoZhat}, the sequences $(k_n)_{n\ge 2}$ such that $k_n=0$ for $w$-almost every $n$ form a closed subgroup of $\hat{\mathbb Z}$, namely, $\pi(Z_w)$, which then coincides with the kernel of the projection $\hat{ \mathbb Z}\to \hat{\mathbb Z}/w$. By a standard fact about profinite groups (see e.g.~\cite[Theorem~1.2.5]{wilson}), $\hat{\mathbb Z}/w$  with the quotient topology induced by this projection is profinite. We will see in \Cref{thm:charred} that the converse is also true, namely, that $\ker(\sigma_w)$ is closed if and only if $w$ is self-divisible.
\end{enumerate}
\end{remark}
\begin{corollary}\label{co:addmax}
The quotient $\beta \mathbb Z/\mathord{\sequiv w}$ may be identified with a subgroup of the ultraproduct $\prod_w \mathbb Z/n \mathbb Z$, which is isomorphic to a quotient (as abstract groups) of $\hat{\mathbb Z}$. If $w$ is self-divisible, then it is isomorphic to  $\prod_{p \in \mathbb P} G_p$, where $G_p= \mathbb Z_p$ if $\phi_w(p)=\omega$, and $G_p= \mathbb Z/p^{\phi_w(p)}\mathbb Z$ otherwise.
\end{corollary}
\begin{proof}
This follows at once from \Cref{rem:vvprime}, \Cref{fact:sbgrpsofzhat}, and the fact that $\hat{\mathbb Z}\cong \prod_{p\in \mathbb P} \mathbb Z_p$.
\end{proof}
\begin{proposition}
 The map $\sigma_w$ is injective if and only if it is an isomorphism, if and only if $w\in \mathrm{MAX}\setminus\set 0$.
\end{proposition}
\begin{proof}
If $w=0$ the map $\sigma_w$ is not defined, so let $w\ne 0$.
Assume $w\in \mathrm{MAX}$, and observe that $\pi\inverse(\set 0)=\mathrm{MAX}$. So if $\pi(v)\ne 0$ then $v\notin \mathrm{MAX}$, hence $w\nsmid v$ by \Cref{fact:maxsdiv}. This shows that, if $\pi(v)\ne 0$, then $\sigma_w(\pi(v))\ne 0$, so $\sigma_w$ is injective. Conversely, if $w\notin \mathrm{MAX}$, there must be $n>1$ such that $(n \mathbb Z)^\cc\in w$. If $n=p_0^{k_0}\cdot\ldots\cdot p_\ell^{k_\ell}$, then $(n \mathbb Z)^\cc= (p_0^{k_0}\mathbb Z)^\cc\cup \ldots \cup(p_\ell^{k_\ell}\mathbb Z)^\cc$. Without loss of generality $(p_0^{k_0}\mathbb Z)^\cc\in w$. Take any $v$ congruent to $p_0^{k_0-1}$ modulo every power of $p_0$ and divided by every other prime power; in other words, take $v$ with $\varphi_v(p_0)=k_0-1$ and $\varphi_v(p')=\omega$ for $p'\ne p_0$. Then $\pi(v)\ne 0$, but $\sigma_w(\pi(v))=\rho_w(v)=0$ because $D(v)\supseteq (p_0^{k_0} \mathbb Z)^\cc \in w$.
\end{proof}

\begin{remark}
The equivalence relation $\sequiv w$ is a congruence with respect to $\odot$ by \cite[Theorem~5.7(a)]{sobot_congruence_2021} . We leave it to the reader to check that everything above in this section works for $(\beta \mathbb Z, \oplus, \odot)$, with $\hat {\mathbb Z}$ viewed as a ring,  where closed subgroups are replaced by closed ideals, etc.
\end{remark}
We already saw several different characterisations of  self-divisibility. In order to provide more, we recall a fact from the theory of profinite groups and make an easy observation.
\begin{fact}\label{fact:autcont}
  If $G$ is a profinite group and $\sigma\from \hat {\mathbb Z}\to G$ is a surjective homomorphism (of abstract groups), then it is automatically continuous.
\end{fact}
\begin{proof}[Proof sketch]
It is enough  to show that if $U$ is an open subgroup of $G$, then $\sigma\inverse(U)$ is open in $\hat{\mathbb Z}$. By compactness, open subgroups have finite index, hence,  it suffices to show that every finite index subgroup of $\hat{\mathbb Z}$ is open. This is in fact true of every topologically finitely generated profinite group by a deep result of Nikolov and Segal~\cite{Nikolov_Segal_2007}, but this special case has a quick proof, which we provide for the sake of completeness. Namely, if $H$ has index $n$ in $\hat {\mathbb Z}$, then $n \hat {\mathbb Z}\subseteq H$, so $H$ can be partitioned into cosets of $n\hat{\mathbb Z}$, hence it  suffices to show that $n\hat {\mathbb Z}$ is open. But $n \hat{\mathbb Z}$ is easily checked to be closed and of finite index, which is equivalent to being open.
\end{proof}

\begin{proposition}\label{co:chartensorpair}
  If $u\tmid v$ and $v\smid t$, then $u\smid t$.
\end{proposition}
\begin{proof}
 Assume $u\tmid v$ and $v\smid t$. Let $(b, c)\models v\otimes t$ and let $a\models u$ be such that $a\mid b$. Then $b\mid c$ and thus $a\mid c$, but $\abs a\leq \abs b$ and thus $u\smid t$ by \Cref{fact:abstensorpair}.
\end{proof}

\begin{theorem}\label{thm:charred}
  The following are equivalent for $w\in \beta \mathbb Z\setminus \set 0$.
  \begin{enumerate}
          \item\label{point:wdivw} The ultrafilter $w$ is self-divisible. 
  \item\label{point:AEab} For all $B\in w$ there is $A\in w$ such that for all $a,a'\in A$ there is $b\in B$ with  $b\mid \operatorname{gcd}(a,a')$.
  \item\label{point:AEbZ} For all $B\in w$ there are $A\in w$ and $b\in B$ such that $A\subseteq b \mathbb  Z$.
  \item\label{point:bb} For all $B\in w$ there is $b\in B$ such that $b \mathbb Z\in w$.
  \item\label{point:bbz} For all $B\in w$ we have $\set{b\in B: b \mathbb  Z\in w}\in w$.

      \item \label{point:kwdf}  For all $k\in \mathbb Z\setminus \set 0$ we have that $kw$ is self-divisible.
      \item \label{point:wnwm} There are $n\ne m$ such that $w^{\oplus n}\sequiv ww^{\oplus m}$.\footnote{The notation is probably self-explicative, but: $w^{\oplus n}:=\overbrace{w\oplus\dots\oplus w}^{\tiny{n \ \text{times}}}$.}
      \item \label{point:vdivw} For all  $v$,  if $w\equiv_{v} 0$ then $w\sequiv v 0$. 
      \item \label{point:dadw} If  ${}^{\ast}\Z\ni a\models w$, then $\{b\in{}^{\ast}\Z: b\mid a\}\subseteq{}^{\ast}D(w)$.
       \item \label{point:cancmax} $Z_w$ is closed under $\oplus$ and, whenever $v\in \mathrm{MAX}$, if $u\oplus v\oplus t\in Z_w$ then $u\oplus t\in Z_w$
  \item \label{point:preimm} $Z_w$ is closed under $\oplus$ and $Z_w=\pi\inverse(\pi(Z_w))$.
  \item \label{point:deprem} $Z_w$ is closed under $\oplus$ and whether $w\tmid u$ only depends on the remainder classes of $u$ modulo standard $n$.
  \item \label{point:kernelclosed} The kernel $\ker(\sigma_w)$ is closed in $\hat{\mathbb Z}$.
   \item \label{point:strqprocyclic} $\beta \mathbb Z/\mathord{\sequiv w}$ is a procyclic group with respect to the quotient topology.\footnote{A \emph{procyclic} group is a profinite group with a dense cyclic subgroup.  Equivalently,  up to isomorphism, it is a quotient of  $\hat{\mathbb Z}$ by a closed subgroup.\label{footnote:procyclic}}
   \item \label{point:strqprofinite} $\beta \mathbb Z/\mathord{\sequiv w}$ is a profinite group with respect to some topology.            \item\label{point:explicitquotient} We have  $(\beta \mathbb Z, \oplus)/\mathord{\sequiv w}\cong\prod_{p \in \mathbb P} G_p$, where $G_p= \mathbb Z_p$ if $\phi_w(p)=\omega$, and $G_p= \mathbb Z/p^{\phi_w(p)}\mathbb Z$ otherwise.
  \end{enumerate}
\end{theorem}
\begin{proof}
  The implication $\eqref{point:wdivw}\IMP\eqref{point:bbz}$ is proven by observing that $w$ is self-divisible if and only if for every $B\in w$ we have $D(w)\cap B\neq\emptyset$, and $\eqref{point:bbz}\IMP\eqref{point:bb}\IMP\eqref{point:AEbZ}\IMP\eqref{point:AEab}$ are immediate.
  
  We now prove $\eqref{point:AEab}\Rightarrow\eqref{point:wdivw}$. By assumption and  the transfer principle, for every $B\in w$ there exists $A\in w$ such that, for every $a, a'\in{}^{\ast}A$, there exists $b\in{}^{\ast}B$ dividing $\operatorname{gcd}(a, a')$. Fix two realisations $a, a'\models w$. Since for every $A\in w$ we have  $a, a'\in{}^{\ast}A$, by assumption for every $B\in w$ there exists $b\in{}^{\ast}B$ such that $b\mid\operatorname{gcd}(a, a')$. By compactness and saturation we can therefore find $b\models w$ such that $b\mid\operatorname{gcd}(a, a')$ and, by \Cref{main:thm}, this gives \eqref{point:wdivw}.
  
Also $\eqref{point:wdivw}\IFF\eqref{point:kwdf}$ follows from the fact that, for every $k\in\Z$, we have $(a, a')\models w\otimes w$ if and only if $(ka, ka')\models kw\otimes kw$. Observe also that, for every $k\in\Z\setminus\{0\}$, we have $kw\sequiv w w^{\oplus k}$. Since \eqref{point:kwdf} implies, by \Cref{co:chartensorpair},  that $kw\sequiv w 0$, we conclude that $\eqref{point:kwdf}\Rightarrow\eqref{point:wnwm}$ by transitivity of $\sequiv w$. 

We prove $\eqref{point:wnwm}\Rightarrow\eqref{point:wdivw}$.  Assume there are $(a, b)\models w\otimes w$ such that $a\nmid b$. By the transfer principle, there exist $p\in{}^{\ast}\mathbb{P}$ and $\alpha\in{}^{\ast}\N$ such that $p^\alpha\mid a$ but $p^\alpha\nmid b$. Notice that $p$ cannot be finite, since a power of a finite prime dividing $a$ must also divide $b$. But then $p^\alpha\nmid kb$ for every $k\in\Z$, and in particular $kw\nsequiv w 0$. Since by \eqref{point:wnwm} there exist $n<m$, such that $w^{\oplus n}\sequiv w w^{\oplus m}$, and as we already observed $kw\sequiv w w^{\oplus k}$, we conclude that $0\sequiv w w^{\oplus(m-n)}\sequiv w(m-n)w$, a contradiction.
  
  Taking $v=w$ yields $\eqref{point:vdivw}\Rightarrow\eqref{point:wdivw}$. Conversely, assuming \eqref{point:wdivw}, if $w\equiv_v 0$, by \Cref{co:chartensorpair} and self-divisibility of $w$ we obtain $w\sequiv v 0$, obtaining \eqref{point:vdivw}.
  
  To see $\eqref{point:wdivw}\IFF\eqref{point:dadw}$, notice that $D(w)=\{n\in\Z: n\mid a\}$, so $\{b\in{}^{\ast}\Z: b\mid a\}\subseteq{}^{\ast}D(w)$ if and only if $a\mid{}^{\ast}a$, which is the nonstandard characterisation of being self-divisible, see \Cref{rem:dfsmid}
  
  We now  prove that $\eqref{point:wdivw}\Rightarrow\eqref{point:deprem}\Rightarrow\eqref{point:preimm}\Rightarrow\eqref{point:cancmax}\Rightarrow\eqref{point:wdivw}$. The equivalence $\eqref{point:deprem}\IFF\eqref{point:preimm}$ is immediate from the definitions of $Z_w$ and $\pi$.  By \Cref{main:thm} and \Cref{fact:sequiveqrel}, if $w$ is self-divisible then $Z_w$ is closed under $\oplus$, and moreover $w\tmid u\IFF D(u)\in w$. Therefore, whether $w\tmid u$  only depends on the finite integers dividing $u$, and this yields $\eqref{point:wdivw}\Rightarrow\eqref{point:deprem}$.
  
  In order to prove $\eqref{point:preimm}\Rightarrow\eqref{point:cancmax}$, recall that, by \Cref{lemma:sdoop}, $\pi$ is an homomorphism and thus, for every $u, v, t\in\beta\Z$, we have  $\pi(u\oplus v\oplus t)=\pi(u)+\pi(v)+\pi(t)$. If $v\in\mathrm{MAX}$, then by \Cref{fact:maxsdiv} $\pi(v)$ is the null sequence, so  $\pi(u\oplus v\oplus t)=\pi(u)+\pi(t)=\pi(u\oplus t)$ and the conclusion follows.
  
In order to  prove that $\eqref{point:cancmax}\Rightarrow\eqref{point:wdivw}$, by \Cref{main:thm} it is enough to show that if \eqref{point:cancmax} holds then  $\equiv_w$ is transitive. Let $u\equiv_w v$ and $v\equiv_w t$, i.e.\ $u\ominus v, v\ominus t\in Z_w$. Since by assumption $Z_w$ is closed under $\oplus$, the ultrafilter $(u\ominus v)\oplus(v\ominus t)=u\oplus(-v\oplus v)\ominus t$ belongs to $Z_w$. But $-v\oplus v\in\mathrm{MAX}$, hence by assumption $u\ominus t\in Z_w$, or equivalently $u\equiv_w t$.

  The implication $\eqref{point:wdivw}\Rightarrow\eqref{point:kernelclosed}$ was proven in \Cref{rem:vvprime}. To prove $\eqref{point:kernelclosed}\Rightarrow\eqref{point:wdivw}$ assume that $\ker\sigma_w$ is closed. By the characterisation of closed subgroups of $\hat{\mathbb Z}$ (\Cref{fact:sbgrpsofzhat}), there is $D\subseteq \mathbb Z$ of the form $\bigcap \left(p^{\phi(p)+1} \mathbb Z\right)^\cc$ such that for all $u\in \beta \mathbb Z$ we have $\pi(u)\in \ker(\sigma_w)$ if and only if $D\subseteq D(u)$. By \Cref{pr:phidf} there is a (possibly principal) self-divisible $v$ such that $D=D(v)$. Since $v$ is self-divisible, by \Cref{main:thm}, for all $u\in \beta \mathbb Z$ we have $D(u)\in v\IFF v\smid u\IFF v\tmid u\IFF D(v)\subseteq D(u)\IFF D\subseteq D(u)\IFF w\smid u\IFF D(u)\in w$, hence $w$ and $v$ contain the same sets of the form $D(u)$. 
  Therefore, $w$ and $v$ contain the same $D(u)^\cc$, hence in particular the same $p^k\mathbb Z$, that is, $D(v)=D(w)$. But, since $v$ is self-divisible, $D(w)=D(v)\in v$, hence $D(w)\in w$.

Recall that, by  \Cref{rem:vvprime}, the topology induced by $\sigma_w$ coincides with the quotient topology (i.e.\ the one induced by $\rho_w$). Then, that  $\eqref{point:kernelclosed}\IMP\eqref{point:strqprocyclic}$ follows from the fact that quotients of procyclic groups by closed subgroups are procyclic (see also the characterisation in \Cref{footnote:procyclic}), and that $\eqref{point:strqprocyclic}\IMP\eqref{point:strqprofinite}$ is obvious. Moreover, if  $\beta \mathbb Z/\mathord{\sequiv w}$ is profinite with respect to some group topology, by \Cref{fact:autcont} the map $\sigma_w$ is automatically continuous, hence its kernel is closed, proving $\eqref{point:strqprofinite}\IMP\eqref{point:kernelclosed}$.

  Finally, $\eqref{point:explicitquotient}\IMP \eqref{point:strqprofinite}$ is clear, and  $\eqref{point:wdivw}\IMP\eqref{point:explicitquotient}$ is \Cref{co:addmax}.
\end{proof}
We take the opportunity to observe that the equivalence of~\ref{point:AEab} above with point \ref{point:mainthmgcd} of \Cref{main:thm} is a special case of~\cite[Theorem~5.23]{LMON}.
\section{Concluding remarks and an open problem}\label{sec:craaop}
Recall that, by \Cref{fact:equivrefl},  $\equiv_w$ is always reflexive. In \Cref{thm:transimpsym} we saw that, whenever $\equiv_w$ is transitive, then it is automatically symmetric. We were not able to determine whether the converse holds.
\begin{problem}\label{prob:symntrans}
  Are there ultrafilters $w\in \beta \mathbb Z \setminus \set 0$ such that $\equiv_w$ is symmetric, but not transitive?
\end{problem}
Our investigation of \v Sobot's congruence relations $\equiv_w$ and $\sequiv w$ led us to introduce self-divisible ultrafilters. The abundance of equivalent forms of self-divisibility (cf.~\Cref{main:thm,thm:charred}) seems to suggest that this and related notions should be investigated further. For instance, one may define $u/v$ as $\set{A:\set{n\in \mathbb Z: nA \in v}\in u}$, and observe that $w$ is self-divisible if and only if  $w/w$ is nonempty, if and only if it is an ultrafilter.
We leave it to future work to explore generalisations, for instance by replacing divisibility with other relations,  and applications to areas such as additive combinatorics or Ramsey theory.

\section*{Acknowledgements}
We thank the anonymous referee for their thorough feedback, that helped improve the paper.

M.~Di Nasso and R.~Mennuni are supported by the Italian research project  PRIN 2017: ``Mathematical Logic: models, sets, computability'' Prot.~2017NWTM8RPRIN and are members of the INdAM research group GNSAGA.

\end{document}